\documentclass[12pt]{article}
\usepackage{amsmath}
\usepackage{amssymb}
\usepackage{amsthm}
\usepackage{verbatim}
\usepackage{mathrsfs}

\newcommand{\D}{{\mathcal D}}

\newcommand{\C}{{\mathcal C}}

\newcommand{\R}[0]{\mathbb R}

\newcommand{\Ds}[0]{\mathcal D}

\newtheorem{Th}{Theorem}[section]
\newtheorem{Lemma}{Lemma}[section]
\newtheorem{Prop}[Lemma]{Proposition}
\newtheorem{Coro}[Th]{Corollary}

\newtheorem{Rem}{Remark}[section]

\begin{document}

\title{On the well-posedness of the Holm-Staley $b$-family of equations}
\author{H. Inci}

\maketitle

\begin{abstract}
In this paper we consider the Holm-Staley $b$-family of equations in the Sobolev spaces $H^s(\R)$ for $s > 3/2$. Using a geometric approach we show that, for any value of the parameter $b$, the corresponding solution map,$u(0) \mapsto u(T)$, is nowhere locally uniformly continuous. 
\end{abstract}

\section{Introduction}\label{section_introduction}

Holm and Staley introduced in \cite{holm_staley} the following family of equations
\begin{equation}\label{b_family}
u_t - u_{xxt} + (b+1) uu_x = b u_x u_{xx} + u u_{xxx}
\end{equation}
or rewritten
\begin{equation}\label{b_rewritten}
u_t + u u_x = (1-\partial_x^2)^{-1} (-b u u_x + (b-3) u_x u_{xx})
\end{equation}
related to shallow water, where
\[
u(t,x) \in \R, \quad (t,x) \in \R^2 
\]
denotes the velocity field and $b \in \R$ is a parameter. For $b=2$ we get the Camassa-Holm equation (see e.g. \cite{4}) and for $b=3$ the Degasperis-Procesi equation (see e.g. \cite{DP}). In his seminal paper \cite{arnold}, Arnold showed that the Euler equation can be interpreted as an equation for a flow on groups of diffeomorphisms of the underlying space. It turns out that quite a few nonlinear evolution equations such as the KdV (see \cite{kdv_euler_equation}) or the Camassa-Holm equation (see \cite{CH_euler_equation}) can be treated by such a geometric approach. Recently, in \cite{escher_kolev}, it has been shown that this is also the case for the Holm-Staley $b$-family. In \cite{geometric}, Constantin used such a geometric approach to get local well-posedness and blow-up results for the Camassa-Holm equation on $\R$. In this paper we will use these methods to prove our results.\\
\noindent
The Cauchy problem for \eqref{b_family} in $H^s(\R)$, $s > 3/2$, with initial value $u_0 \in H^s(\R)$, is to find $u \in C^0([0,T],H^s(\R))$ for some $T > 0$, such that we have the following identity in $H^{s-1}(\R)$
\begin{equation}\label{b_integral}
u(t) = u_0 + \int_0^t (1-\partial_x^2)^{-1} (-b u u_x + (b-3) u u_{xx} ) - uu_ x \; ds
\end{equation}
for all $t \in [0,T]$. Here we regard $u_x u_{xx}$ as an element of $H^{s-2}(\R)$ even if $3/2 < s < 2$ -- see Appendix \ref{appendix_b}. With this in mind and the fact that $H^{s-1}(\R)$ is a Banachalgebra we see that the integrand in \eqref{b_integral} is an element of $C^0([0,T],H^{s-1}(\R))$. Hence \eqref{b_integral} makes sense and we have actually $u \in C^1([0,T],H^{s-1}(\R))$.\\

\noindent
Concerning the well-posedness of \eqref{b_family} we have the following result -- see also \cite{Olver}

\begin{Th}\label{th_local_wellposedness}
Let $s > 3/2$. For any given $u_0 \in H^s(\R)$ there is a $T > 0$ and a unique solution $u \in C^0([0,T],H^s(\R)$ to the Cauchy problem \eqref{b_family} with initial value $u(0)=u_0$. The $T$ can be chosen to be the same in a neighborhood $U \subseteq H^s(\R)$ of $u_0$. Moreover the map 
\[
 U \to C^0([0,T],H^s(\R)),\quad u_0 \mapsto u
\]
is continuous.
\end{Th}

Now one asks whether the map mentionned in the above theorem is more than continuous, e.g. $C^1$ or at least locally lipschitz. We have the following negative answer.

\begin{Th}\label{th_nonuniform}
Denote by $U_T \subseteq H^s(\R)$ the set of initial values for which we have existence up to at least $T$. Then the map  $U_T \to H^s(\R),\; u_0 \mapsto u(T)$, mapping the initial value to the time $T$ value of the solution, is nowhere locally uniformly continuous.  
\end{Th}

\noindent
Results saying that the solution map $u_0 \mapsto u(T)$ has not the property to be uniformly continuous on bounded sets is known. On the circle this was proved in the case $b=2$ (Camassa-Holm equation) in \cite{periodic_CH} and in the case $b=3$ (Degasperis-Procesi equation) in \cite{periodic_DP} for $s \geq 2$. For the $b$-family on the line this was proved in \cite{general_b}.\\

\section{The geometric framework}\label{section_geometric}

In \cite{composition} we considered for $s \in \R$, $s > 3/2$, the space $\Ds^s(\R)$ (cf M. Cantor \cite{cantor}) given by
\begin{eqnarray*}
 \Ds^s(\R) &:=& \big\{ \varphi: \R \to \R \;\; C^1-\mbox{diffeomorphism} \; \big| \; \varphi_x > 0 \mbox{ and } \varphi(x) - x \in H^s(\R)\big\} \\
&=& \big\{ \varphi(x)=x + f(x) \; \big| \; f \in H^s(\R) \mbox{ and } \varphi_x > 0 \big\}
\end{eqnarray*}
where $H^s(\R)$ denotes the space of Sobolev functions on $\R$ of class $s$. In terms of the Fourier transform this reads as (see e.g. \cite{adams})
\[ 
 H^s(\R) = \big\{ f \in L^2(\R) \; \big| \; (1+\xi^2)^{s/2} \hat f(\xi) \in L^2(\R;\C) \big\}
\]
where $\hat f$ is the Fourier transform of $f$. Equipped with the scalar product (taking the real part)
\[
 \langle f, g \rangle_s = \Re \int_\R (1+\xi^2)^s \hat f(\xi) \overline{\hat g(\xi)} d\xi
\]
it becomes a Hilbert space. Then 
\[
 \Ds^s(\R) - id = \big\{ \varphi(x) - x \; \big| \; \varphi \in \Ds^s(\R) \big\} \subseteq H^s(\R)
\]
is open and thus has naturally the structure of a analytic Hilbert manifold (cf e.g. \cite{composition}). Moreover $\Ds^s(\R)$ is a topological group under composition. More precisely, for any $k \in \mathbb{Z}_{\geq 0}$,
\[
 H^{s+k}(\R) \times \Ds^s(\R) \to H^s(\R), (f,\varphi) \mapsto f \circ \varphi 
\]
is a $C^k$-map. In the literature the partial maps of the compostion map are referred to as the $\alpha$- resp. the $\omega$-lemma -- see \cite{EM}.\\
In the following we need the notion of sprays. These are special vectorfields on the tangent bundle -- see e.g. \cite{Lang}. In our case we have the following identification for the tangent bundle of $\Ds^s(\R)$
\[
 T\Ds^s(\R) = \Ds^s(\R) \times H^s(\R).
\]
Thus a spray can be defined by a map $S$ with the following structure
\begin{eqnarray*}
 S:\Ds^s(\R) \times H^s(\R) &\to& H^s(\R) \times H^s(\R) \\
(\varphi,v) &\mapsto& \big(v,\Gamma_\varphi(v,v)\big)
\end{eqnarray*}
where $\Gamma$, called the Christoffel map of the spray $S$, is a map
\begin{eqnarray*}
 \Gamma : \Ds^s(\R) &\to& L\big(H^s(\R),H^s(\R);H^s(\R)\big)\\
\varphi &\mapsto& \Gamma_\varphi(\cdot,\cdot)
\end{eqnarray*}
with values in the continuous $H^s(\R)$-valued bilinear forms on $H^s(\R)$. Since we are just interested in the quadratic form $\Gamma_\varphi(v,v)$ we assume $\Gamma_\varphi$ to be symmetric. The integral curves of $S$ projected on $\Ds^s(\R)$ are called the geodesics of $S$. Like in the case of a Riemannian manifold we have also here the notion of an exponential map -- see e.g. \cite{Lang}. More precisely, the equation of the geodesics reads as
\begin{equation}\label{geodesic_eq}
\varphi_{tt} = \Gamma_\varphi(\varphi_t,\varphi_t)
\end{equation}
where the subscript $t$ denotes differentiation with respect to $t$. For analytic $S$ the Picard iteration gives local solutions of \eqref{geodesic_eq} with initial data $\varphi(0)=id \in \Ds^s(\R)$ and $\varphi_t(0)=v \in H^s(\R)$. Because of the scaling properties of \eqref{geodesic_eq} there exists a neighborhood $V$ of $0 \in H^s(\R)$ such that the initial value problem 
\begin{equation}\label{geodesic_ivp_raw}
 \left\{ 
\begin{array}{c}
\varphi_{tt} = \Gamma_\varphi(\varphi_t,\varphi_t)\\
\varphi(0)=id, \varphi_t(0) = v
\end{array}
\right.
\end{equation}
admits a solution on the time interval $[0,1]$ for all $v \in V$. This allows us to define the exponential map $\exp$ as
\begin{eqnarray*}
\exp: V &\to& \Ds^s(\R)\\
v &\mapsto& \varphi_v(1)
\end{eqnarray*}
where $\varphi_v$ is the solution of \eqref{geodesic_ivp_raw}. Because of the analytic dependence of solutions of \eqref{geodesic_ivp_raw} on the initial values, see \cite{Lang}, we know that $\exp$ is a smooth map. Moreover the derivative of $\exp$ at $0 \in H^s(\R)$ is the identity, i.e.
\[
 d_0 \exp:H^s(\R) \to H^s(\R), \quad v \mapsto v
\] 
where we have identified $T_{id} \Ds^s(\R)$ with $H^s(\R)$. By the inverse function theorem for Banach spaces, see \cite{Lang}), $\exp$ is an analytic diffeomorphism between a neighborhood $U$ of $0 \in H^s(\R)$ and a neighborhood $V$ of $id \in \Ds^s(\R)$, i.e.
\[
 \exp:U \to V
\]
is an analytic diffeomorphism.\\
For our purpose we define $\Gamma$ at $id \in \Ds^s(\R)$ for $v \in H^s(\R)$ with $s >3/2$ by
\begin{equation}\label{gamma_id}
 \Gamma_{id}(v,v) = (1-\partial_x^2)^{-1} (-b v v_x + (b-3)v_x v_{xx})
\end{equation}
which is a continuous $H^s(\R)$-valued quadratic form on $H^s(\R)$. For $\varphi \in \Ds^s(\R)$ arbitrary, $\Gamma_\varphi$ is defined by
\begin{equation}\label{gamma_phi}
\Gamma_\varphi(v,v) = \big(\Gamma_{id}(v \circ \varphi^{-1},v \circ \varphi^{-1})\big) \circ \varphi.
\end{equation}
In view of the poor regularity properties of the composition map, apriori it is not clear if $\Gamma$ defines a smooth spray. In the next section we verify that this is indeed the case. In the following we will make some formal computations to show how the geodesics of $S$ and solutions to \eqref{b_family} are related. Assume that $\varphi:[0,T] \to \Ds^s(\R)$ solves the initial value problem \eqref{geodesic_ivp_raw}. Then we have for $u:=\varphi_t \circ \varphi^{-1}$ 
\begin{eqnarray*}
 \varphi_{tt} &=& (u \circ \varphi)_t = u_t \circ \varphi + u_x \circ \varphi \cdot \varphi_t\\
&=& u_t \circ \varphi + u_x \circ \varphi \cdot u \circ \varphi. 
\end{eqnarray*}
Substituting this expression into equation \eqref{geodesic_eq} we get
\begin{eqnarray*}
u_t \circ \varphi + u_x \circ \varphi \cdot u \circ \varphi &=& \Gamma_\varphi(\varphi_t,\varphi_t)\\
&=& \Gamma_{id}(\varphi_t \circ \varphi^{-1},\varphi_t \circ \varphi^{-1}) \circ \varphi 
\end{eqnarray*} 
or by \eqref{gamma_id} equivalently
\[
 u_t + u u_x = (1-\partial_x^2)^{-1} ( -b u u_x + (b-3)u_x u_{xx})
\]
which is equation \eqref{b_rewritten}. In the next section we show that by this approach one gets local well-posedness results for equation \eqref{b_family}.

\section{Local wellposedness of the $b$-family of equations}\label{section_local_wellposed}

In this section we establish local existence and uniqueness for the Cauchy problem
\begin{equation}\label{cauchy_problem}
u_t + u u_x = (1-\partial_x^2)^{-1} (-b u u_x + (b-3) u u_{xx}),\quad u(0)=u_0 \in H^s(\R)
\end{equation}
in $H^s(\R)$, $s > 3/2$. 

\begin{Th}\label{th_smooth_spray}
The spray $S$  given by
\[
 S:\Ds^s(\R) \times H^s(\R) \to H^s(\R) \times H^s(\R),\quad (\varphi,u) \mapsto \big(v,\Gamma_\varphi(u,u)\big)
\]
where 
\[
\Gamma_\varphi(u,u)= R_\varphi (1-\partial_x^2)^{-1} \big(-b (R_{\varphi^{-1}}u) \cdot (R_{\varphi^{-1}}u)_x + (b-3) (R_{\varphi^{-1}}u)_x \cdot (R_{\varphi^{-1}}u)_{xx} \big) 
\]
is analytic.
\end{Th}

\noindent
Recall that we use the notation $R_\varphi$ for right translation $H^s(\R) \to H^s(\R), f \mapsto f \circ \varphi$. Before proving Theorem \ref{th_smooth_spray} we show the following lemma.

\begin{Lemma}\label{lemma_smooth_conjugation}
Let $s > 3/2$. For $k=1,2$ the map
\[
 \delta^{(k)}:\Ds^s(\R) \times H^s(\R) \to H^{s-k}(\R),\quad (\varphi,f) \mapsto R_\varphi \partial_x^k R_{\varphi^{-1}}f
\]
is analytic.
\end{Lemma}

\begin{proof}
Consider first the case $k=1$. Then we have
\[
 \delta^{(1)}(\varphi,f)=R_\varphi \partial_x R_{\varphi^{-1}} f = \frac{f_x}{\varphi_x}
\]
and this is an analytic expression in $\varphi$ and $f$. Similarly for $k=2$ we have we have
\[
 \delta^{(2)}(\varphi,f)=R_\varphi \partial_x^2 R_{\varphi^{-1}} f = \frac{f_{xx}}{\varphi_x} - \frac{f_x \cdot \varphi_{xx}}{(\varphi_x)^3}
\]
which also holds in the case $s < 2$ -- see Appendix \ref{appendix_b} for the conventions in this case. We see from the expressions for $\delta^{(2)}(\varphi,f)$ that it is indeed analytic.
\end{proof}

\noindent
Now we can give the proof Theorem \ref{th_smooth_spray}.

\begin{proof}[Proof of Theorem \ref{th_smooth_spray}]
Consider the continuous symmetric bilinear map
\begin{eqnarray*}
B:H^s(\R) \times H^s(\R) &\to& H^{s-2}(\R)\\
(u,v) &\mapsto& \frac{1}{2} \big(-b u v_x - b v u_x \big) + \frac{1}{2} \big((b-3) u v_{xx} + (b-3) v u_{xx}\big).
\end{eqnarray*}
That the range of this expression is in $H^{s-2}(\R)$ and its continuity follow from the Banach algebra properties of $H^s(\R)$ -- see e.g. \cite{adams} (for $s<2$ see Appendix \ref{appendix_b}). From Lemma \ref{lemma_smooth_conjugation} we know that the map
\begin{eqnarray*}
\Ds^s(\R) &\to& L\left(H^s(\R), H^s(\R);H^{s-2}(\R)\right)\\
\varphi &\mapsto& \left[(u,v) \mapsto R_\varphi B(R_{\varphi^{-1}}u,R_{\varphi^{-1}}v) \right]
\end{eqnarray*}
is analytic. Again using Lemma \ref{lemma_smooth_conjugation} we get that
\begin{eqnarray*}
\Ds^s(\R) &\to& L\left(H^s(\R);H^{s-2}(\R)\right)\\
\varphi &\mapsto& \left[ u \mapsto R_\varphi (1-\partial_x^2)R_{\varphi^{-1}}u \right]
\end{eqnarray*}
is analytic. Hence
\begin{eqnarray*}
\Ds^s(\R) &\to& L\big(H^s(\R),H^{s-2}(\R)\big)\\
\varphi &\mapsto& \left[ u \mapsto R_\varphi (1-\partial_x^2) R_{\varphi^{-1}} u \right]
\end{eqnarray*}
is analytic. Note that $u \mapsto R_\varphi (1-\partial_x^2) R_{\varphi^{-1}}u$ is invertible with inverse given by $v \mapsto R_\varphi (1-\partial_x^2)^{-1} R_{\varphi^{-1}}v$. Now as inversion of linear operators is an analyitc process we see that the map
\begin{eqnarray*}
\Ds^s(\R) &\to& L\big(H^{s-2}(\R),H^s(\R)\big)\\
\varphi &\mapsto& \left[v \mapsto R_\varphi (1-\partial_x^2)^{-1} R_{\varphi^{-1}} v \right]
\end{eqnarray*}
is analytic. Piecing the maps together we get from the identity
\begin{eqnarray*}
 \Gamma_\varphi(u,v) &=& R_\varphi (1-\partial_x^2)^{-1} B(R_{\varphi^{-1}}u,R_{\varphi^{-1}}v)\\
&=& R_\varphi (1-\partial_x^2)^{-1} R_{\varphi^{-1}} R_\varphi B(R_{\varphi^{-1}}u,R_{\varphi^{-1}}u)
\end{eqnarray*}
that $\Gamma:\Ds^s(\R) \to L\left(H^s(\R),H^s(\R);H^s(\R)\right)$ is analytic. This completes the proof of the theorem.
\end{proof}

\noindent
Now consider the initial value problem (IVP)
\begin{equation}\label{geodesic_ivp}
\left\{
\begin{array}{c}
\varphi_{tt} = \Gamma_\varphi(\varphi_t,\varphi_t) \\
\varphi(0)=id \in \Ds^s(\R),\quad \varphi_t(0) = u_0 \in H^s(\R).
\end{array}
\right.
\end{equation}

\noindent
The Picard theorem gives us local solutions to the IVP \eqref{geodesic_ivp}. With this and the discussion at the and of section \ref{section_geometric} we get the following local existence result.

\begin{Lemma}\label{lemma_local_existence}
Let $s > 3/2$. Given $u_0 \in H^s(\R)$, there exists $u \in C^0\big([0,T],H^s(\R)\big)$, for some $T > 0$, such that
\begin{equation}\label{strong_solution}
 u(t)= u_0 + \int_0^t (1-\partial_x^2)^{-1} (-b u u_x + (b-3) u_x u_{xx}) - u u_x
\end{equation}
holds for all $t \in [0,T]$.
\end{Lemma}

\begin{proof}
Consider the IVP \eqref{geodesic_ivp}. Since $\Gamma$ is smooth, there exists $T > 0$, such that we have a
\[
 \varphi \in C^\infty\big([0,T],\Ds^s(\R)\big)
\]
solving \eqref{geodesic_ivp}. We claim that $u:=\varphi_t \circ \varphi^{-1}$ is a solution to \eqref{strong_solution}. From \cite{composition} we know that $u \in C^0([0,T],H^s(\R))$. We also have $u(0)=\varphi_t(0)=u_0$. From the Sobolev imbedding theorem we see that $u$ resp. $\varphi$ are in $C^1([0,T] \times \R)$. Taking the $t$-derivative of $u \circ \varphi$ we get
\[
 u_t \circ \varphi + u_x \circ \varphi \cdot \varphi_t = u_t \circ \varphi + u_x \circ \varphi \cdot u \circ \varphi.
\]
On the other hand we have
\[
 (u \circ \varphi)_t = \varphi_{tt} = \Gamma_\varphi(\varphi_t,\varphi_t)
\]
as functions on $[0,T] \times \R$. Thus we get
\[
 u_t \circ \varphi + u_x \circ \varphi \cdot u \circ \varphi = \Gamma_\varphi(\varphi_t,\varphi_t)
\]
or
\begin{eqnarray*}
 u_t &=& \Gamma_\varphi(\varphi_t,\varphi_t) \circ \varphi^{-1} - u u_x\\
&=& (1-\partial_x^2)^{-1} (-b u u_x + (b-3)u_x u_{xx}) - u u_x
\end{eqnarray*}
as functions on $[0,T] \times \R$. As both sides are continuous functions, we get by the fundamental lemma of calculus for $t \in [0,T]$
\begin{equation}\label{pointwise}
 u(t)=u_0 + \int_0^t (1-\partial_x^2)^{-1} (-b u u_x + (b-3)u_x u_{xx}) - u u_x.
\end{equation}
But as the integrand is in $C^0([0,T],H^{s-1}(\R))$, the identity \eqref{pointwise} holds also in $H^{s-1}(\R)$.
\end{proof}

To get uniqueness we use the fact that for $u \in C\left([0,T];H^s(\R)\right)$ there is a unique flow $\varphi \in C^1\left([0,T];\D^s(\R)\right)$, i.e. $\varphi$ solving
\[
 \varphi_t = u \circ \varphi,\quad \varphi(0)=id
\]
-- see \cite{thesis}. With this we will prove
\begin{Lemma}\label{lemma_uniqueness}
Let $s > 3/2$. Assume that we have two solutions $u,w \in C^0\big([0,T],H^s(\R)\big)$ to the Cauchy problem \eqref{b_integral} with $u(0)=w(0)=u_0 \in H^s(\R)$. Then we have actually $u=w$ on $[0,T]$. 
\end{Lemma}

\begin{proof}
Consider the flows $\varphi$ resp. $\psi$ in $C^1\big([0,T],\Ds^s(\R)\big)$ corresponding to $u$ resp. $w$ as discussed above. We will show that $\varphi$ resp. $\psi$ are geodesics. By the uniqueness of geodesics with the same initial condition we will get $\varphi=\psi$ resp. $\varphi_t \circ \varphi^{-1} = \psi_t \circ \psi^{-1}$ or equivalently $u=w$. Now consider $u \circ \varphi$ which is $C^1$ in $[0,T] \times \R$. Taking the $t$-derivative we get pointwise
\[
 \frac{d}{dt} (u \circ \varphi) = u_t \circ \varphi + u_x \circ \varphi \cdot u \circ \varphi.
\]
Since $u$ is a solution of the Cauchy problem we have
\begin{eqnarray*}
 \frac{d}{dt} (u \circ \varphi) &=& R_\varphi \big((1-\partial_x^2)^{-1} (-b u u_x + (b-3)u_x u_xx)\big)\\
&=& \Gamma_\varphi(\varphi_t \circ \varphi^{-1},\varphi_t \circ \varphi^{-1})
\end{eqnarray*}
where $\Gamma$ is as in \eqref{gamma_phi}. From the fundamental lemma of calculus we get
\[
 \varphi_t(t)=u_0 + \int_0^t \Gamma_\varphi(\varphi_t \circ \varphi^{-1},\varphi_t \circ \varphi^{-1}).
\]
But as the integrand is in $C^0\big([0,T],H^s(\R)\big)$ we see that $\varphi \in C^2\big([0,T],\Ds^s(\R)\big)$ and that it is a geodesic. Hence the claim.
\end{proof}

Using Lemma \ref{lemma_local_existence} and Lemma \ref{lemma_uniqueness} we get 

\begin{Th}\label{th_local_well_posedness}
Let $s > 3/2$. The Cauchy problem \eqref{b_integral} is locally well-posed in $H^s(\R)$, i.e. given $u_0 \in H^s(\R)$ there exists a solution $u \in C^0\big([0,T],H^s(\R)\big)$ for some $T > 0$. Further $u$ is unique on $[0,T]$ and the correspondence $u_0 \to u$ is continuous.
\end{Th}

\begin{Rem}
The correspondence $u_0 \to u$ is meant as a map
\[
 U_T \to C^0\big([0,T],H^s(\R)\big),\quad u_0 \mapsto u
\]
\end{Rem}

This result is not new. This was e.g. done in \cite{Olver}. They looked at a regularized version of \eqref{cauchy_problem} and took the limit. Also in \cite{Rodriguez} there is a similar result. They use Kato's abstract semigroup method. The same method is also used in \cite{periodic_DP} for the periodic case. We used the geometric setting as was e.g. done in \cite{geometric} to achieve local well-posedness via the Picard theorem.

\section{Non-uniform dependence of the solution map}\label{section_nonuniform}

In this section we will prove the non-uniform dependence of the solution map, i.e. we will prove Theorem \ref{th_nonuniform}. Recall that we denoted for $T > 0$ the set $U_T \subseteq H^s(\R)$ to be those $u_0$ for which we have existence beyond $T$. Note also that we have the following scaling property for equation \eqref{b_family}: If $u$ is a solution then for $\lambda > 0$ 
\[
 v(x,t):=\lambda u(x,\lambda t)
\]
is also a solution and $U_{\lambda T}=\frac{1}{\lambda} U_T$. Therefore it suffices to consider the case $T=1$ to prove Theorem \ref{th_nonuniform}. Hence the theorem will follow from

\begin{Prop}\label{prop_nonuniform}
Let $s > 3/2$ and $U:=\left. U_T \right|_{T=1}$. Denote by $\Phi$ the time-one map $\Phi:U \to H^s(\R), u_0 \mapsto u(1)$ for the Cauchy problem \eqref{b_family}. Then $\Phi$ is nowhere locally uniformly continuous, i.e. for any non-empty $V \subseteq U$ the restriction $\left. \Phi \right|_V$ is not uniformly continuous.
\end{Prop}  

To prove Proposition \ref{prop_nonuniform} we will use a conserved quantity (cf \cite{escher_kolev}, Proposition 9). Consider equation \eqref{b_rewritten} and let $u$ be a solution, $\varphi$ its corresponding flow. Then we have, omitting the arguments $(t,x)$,
\begin{Lemma}\label{lemma_conserved}
Let $y:=(1-\partial_x^2)u$. Then we have for all $t$ the following identity in $H^{s-2}(\R)$
\begin{equation}\label{conserved} 
 y \circ \varphi \cdot (\varphi_x)^b = y(t=0)
\end{equation}
\end{Lemma}

\begin{proof}
Taking the $t$-derivate of $y \circ \varphi$ we get
\begin{multline}\label{w_t}
\frac{d}{dt} y \circ \varphi = \varphi_{tt} - \frac{\varphi_{ttxx}}{\varphi_x^2} + 2 \frac{\varphi_{txx} \cdot \varphi_{tx}}{\varphi_x^3} + \frac{\varphi_{ttx} \cdot \varphi_{xx}}{\varphi_x^3}\\ + \frac{\varphi_{tx} \cdot \varphi_{txx}}{\varphi_x^3} - 3 \frac{\varphi_{tx} \cdot \varphi_{xx} \cdot \varphi_{tx}}{\varphi_x^4}
\end{multline}
On the other hand equation \eqref{geodesic_ivp} gives
\begin{multline*}
R_\varphi (1-\partial_x^2)(\varphi_{tt} \circ \varphi^{-1}) = \\ R_\varphi \big(-b (\varphi_t \circ \varphi^{-1}) \cdot (\varphi_t \circ \varphi^{-1})_x + (b-3) (\varphi_t \circ \varphi^{-1})_x \cdot (\varphi_t \circ \varphi^{-1})_{xx}\big)
\end{multline*}
Expanding this equation we get
\begin{multline}\label{geodesic_expanded}
\varphi_{tt} - \frac{\varphi_{ttxx}}{\varphi_x^2} + \frac{\varphi_{ttx} \cdot \varphi_{xx}}{\varphi_x^3} = -b \frac{\varphi_t \cdot \varphi_{tx}}{\varphi_x}
+ (b-3) \frac{\varphi_{tx}}{\varphi_x} \left( \frac{\varphi_{txx}}{\varphi_x^2} - \frac{\varphi_{tx} \cdot \varphi_{xx}}{\varphi_x^3}\right)
\end{multline}

Hence from \eqref{w_t}
\begin{multline}\label{total_derivative}
\frac{d}{dt} \left[ (R_\varphi y) \cdot \varphi_x^b \right] = \Big[ \varphi_{tt} - \frac{\varphi_{ttxx}}{\varphi_x^2} + 3 \frac{\varphi_{txx} \cdot \varphi_{tx}}{\varphi_x^3} + \frac{\varphi_{ttx} \cdot \varphi_{xx}}{\varphi_x^3}- 3 \frac{\varphi_{tx}^2 \cdot \varphi_{xx}}{\varphi_x^4} \Big] \cdot \varphi_x^b\\ + \Big[ \varphi_t - \frac{\varphi_{txx}}{\varphi_x^2} + \frac{\varphi_{tx} \cdot \varphi_{xx}}{\varphi_x^3} \Big] \cdot b \varphi_x^{b-1} \varphi_{tx}
\end{multline}
Combining \eqref{total_derivative} and \eqref{geodesic_expanded} we get
\[
 \frac{d}{dt} \left[ (R_\varphi y) \cdot \varphi_x^b \right] = 0
\]
hence $[0,T] \to H^{s-2}(\R),t \mapsto y \circ \varphi \cdot \varphi_x^b$ is constant, i.e. \eqref{conserved} holds.
\end{proof}

\noindent

As $1-\partial_x^2:H^s(\R) \to H^{s-2}(\R)$ is an isomorphism, it will be enough to establish that $y_0 \mapsto y(1)$ is nowhere locally uniformly continuous in order to show Proposition \ref{prop_nonuniform}. We will use \eqref{conserved} in the form
\begin{equation}\label{conserved_rewritten}
y(1) = \left(\frac{y_0}{\varphi_x(1)^b}\right) \circ \varphi(1)^{-1}
\end{equation}
The approach is as in \cite{solution_map}. The idea is to produce a slight change $\tilde \varphi(1)$ so that $\tilde y_0$ doesn't change much, but $\tilde y(1)$ does. Since composition behaves bad this works. To make such perturbations we will employ the properties of the exponential map. But first we have to establish some preleminary lemmas.

\begin{Lemma}\label{lemma_composition_bounded_below_above}
Given $\varphi_\bullet \in \Ds^s(\R)$ there exists a neighborhood $W$ of $\varphi_\bullet$ in $\Ds^s(\R)$, such that for some constant $C>0$ we have
\[
\frac{1}{C} ||y||_{s-2} \leq ||R_\varphi^{-1} (y/\varphi_x^b)||_{s-2} \leq C ||y||_{s-2}
\]
for all $y \in H^{s-2}(\R)$ and $\varphi \in W$.
\end{Lemma}

\begin{proof}
First we establish the second inequality. For $\varphi_n \to \varphi_\bullet$ in $\Ds^s(\R)$ we have (see Remark \ref{remark_weak_continuity}) $R_{\varphi_n}^{-1}(y/\varphi_x^b)$ converges weakly to $y$ in $H^{s-2}(\R)$. By the uniform boundedness principle we get a neighborhood $W_1$ of $\varphi_\bullet$ in $\Ds^s(\R)$ and $C_1 > 0$ such that
\[
 ||R_\varphi^{-1} (y/\varphi_x^b) ||_{s-2} \leq C ||y||_{s-2}
\]
for all $\varphi \in W_1$ and $y \in H^{s-2}(\R)$. Now consider the first inequality. As in the first part we get a $W_2$ and $C_2 >0$ with
\[
 ||\varphi_x^b R_\varphi y||_{s-2} \leq C_2 ||y||_{s-2}
\]
for all $\varphi \in W_2$ and $y \in H^{s-2}(\R)$. Taking for $y$ the expression $R_\varphi^{-1} (y/\varphi_x^b)$ gives the first inequalty.
\end{proof}

\begin{Lemma}\label{lemma_dexp}
Let $s > 3/2$. The set of $u_0 \in U \cap H^{s+1}$ with $d_{u_0} \exp \neq 0$ is dense in $U$.
\end{Lemma}

\begin{proof}
As $d_0 \exp$ is the identity map and $u_0 \mapsto d_{u_0} \exp$ is analytic the claim follows immediately.
\end{proof}

Now we can give a proof for Proposition \ref{prop_nonuniform}

\begin{proof}[Proof of Proposition \ref{prop_nonuniform}]
Take $u_0 \in U$. We will show that $\Phi$ is not uniformly continuous on any neighborhood of $u_0$. As easily seen we can restrict ourselves to a dense subset of $U$. So we can assume $u_0 \in H^{s+1}$ and $d_{u_0} \exp \neq 0$ by Lemma \ref{lemma_dexp}. In particular we can fix $v \in H^s(\R)\smallsetminus \{0\}$ and $x_0 \in \R$ with
\[
 \left(d_{u_0}\exp(v)\right)(x_0) > m ||v||_s,\quad m > 0
\] 
First we choose $R_1 > 0$ such that Lemma \ref{lemma_composition_bounded_below_above} holds for $\varphi_\bullet=\exp(u_0)$ in the ball $B_{R_1}(u_0)$, i.e.
\begin{equation}\label{ineq_below_above}
\frac{1}{C_1} ||y||_{s-2} \leq ||R_\varphi^{-1} (y/\varphi_x^b)||_{s-2} \leq C_1 ||y||_{s-2}
\end{equation}
for all $y \in H^{s-2}(\R)$ and $\varphi \in B_{R_1}(u_0)$. Taking $R_2 \leq R_1$ we can ensure additionally
\[
 ||R_\varphi^{-1} y||_{s-2} \leq C_2  ||y||_{s-2}
\]
for all $y \in H^{s-2}(\R)$ and $\varphi \in B_{R_2}(u_0)$. By choosing $R_3 \leq R_2$ we can establish the conditions of Lemma \ref{lemma_lipschitz} and \ref{lemma_lipschitz_inv} for all $\varphi \in B_{R_3}$ where in the following we denote the constant appearing in both lemmas by $C_3$. Further we denote by $C > 0$ the constant from the Sobolev imbedding
\[
 ||f||_\infty \leq C ||f||_s
\]
Take arbitrary $w,h \in H^s(\R)$ with $w, w+h \in B_{R_3}(u_0)$ and consider the Taylor expansion
\[
 \exp(w+h)=\exp(w)+d_w \exp(w)+\int_0^1 (1-t) d_{w+th}^2\exp(h,h) dt
\]
Choosing $0 < R_4 \leq R_3$ we can garantuee 
\[
 ||d_{w}^2 \exp(h,h)||_s \leq K ||h||_s^2
\]
\[
 ||d_{w_1}^2 \exp(h,h)-d_{w_2}^2 \exp(h,h)||_s \leq K ||w_1-w_2||_s ||h||_s^2
\]
for all $w,w_1,w_2 \in B_{R_4}(u_0)$ and some constant $K > 0$. By further decreasing $R_5 \leq R_4$ we can ensure $\max\{C \cdot K \cdot R_5,C \cdot K \cdot R_5^2\} < m/2$. Finally by choosing $R_\ast \leq R_5$ we can ensure
\[
 |\varphi(x)-\varphi(y)| \leq L |x-y|
\]
for all $\varphi \in \exp(B_{R_\ast}(u_0))$. The goal is now to prove that $\Phi$ is not uniformly continuous on $B_R(u_0)$ for $0 < R \leq R_\ast$. So we fix $R \leq R_\ast$. In order to apply Lemma \ref{lemma_vanishing_support1} resp. Lemma \ref{lemma_vanishing_support2} we define the sequence of numbers
\[
 r_n=\frac{m}{8n}||v||_s,\quad n \geq 1 
\]
and choose $w_n \in C_c^\infty(\R)$ with support in $(x_0-\frac{1}{L}r_n,x_0+\frac{1}{L}r_n)$ and $||w_n||_s=R/4$. Further we define $v_n:=v/n$ and let $N \geq 1$ such that $||v_n||_s \leq R/4$ for $n \geq N$. With this preliminary work we define for $n \geq N$ two sequences of initial data:
\[
 x_n = u_0+w_n \quad \mbox{and} \quad \tilde x_n = x_n + v_n = u_0 + w_n + v_n 
\]
We clearly have $x_n,y_n \in B_R(u_0)$ for $n \geq N$ and $||x_n-\tilde x_n||_s \to 0$ for $n \to \infty$. Correspondingly we define
\[
 \varphi_n = \exp(x_n) \quad \mbox{and} \quad \tilde \varphi_n = \exp(\tilde x_n)
\]
We claim that $\limsup_{n \to \infty} ||\Phi(x_n)-\Phi(\tilde x_n)||_s > 0$. With $y_n=(1-\partial_x^2)x_n$ and $\tilde y_n=(1-\partial_x^2)\tilde x_n$ and using the conservation law \eqref{conserved_rewritten} it is enough to prove
\[
 \limsup_{n \to \infty} ||R_{\varphi_n}^{-1}\left(y_n/(\varphi_n)_x^b\right)-R_{\tilde \varphi_n}^{-1} \left(\tilde y_n/(\tilde \varphi_n)_x^b \right)||_{s-2} > 0
\]
We consider the parts of $y_n,\tilde y_n$ seperately
\[
 y_n=(1-\partial_x^2)(u_0 + w_n) \quad \mbox{and} \quad \tilde y_n=(1-\partial_x^2)(u_0+w_n+v_n)
\]
For the $u_0$-part we have, denoting $y_0=(1-\partial_x^2)u_0 \in H^{s-1}$,
\begin{multline*}
 ||R_{\varphi_n}^{-1} (y_0/(\varphi_n)_x^b) - R_{\tilde \varphi_n}^{-1} (y_0/(\tilde \varphi_n)_x^b||_{s-2} \leq ||R_{\varphi_n}^{-1} (y_0/(\varphi_n)_x^b) - R_{\varphi_n}^{-1} (y_0/(\tilde \varphi_n)_x^b)||_{s-2} \\
+ ||R_{\varphi_n}^{-1} (y_0/(\tilde \varphi_n)_x^b) - R_{\tilde \varphi_n}^{-1} (y_0/(\tilde \varphi_n)_x^b)||_{s-2}
\end{multline*}
The first term on the right can be estimated by
\[
||R_{\varphi_n}^{-1} (y_0/(\varphi_n)_x^b) - R_{\varphi_n}^{-1} (y_0/(\tilde \varphi_n)_x^b)||_{s-2} \leq C_2 ||y_0/(\varphi_n)_x^b-y_0/(\tilde \varphi_n)_x^b||_{s-2}
\]
The latter goes to $0$ as $n \to \infty$ as dividing by $\varphi_x^b$ is an analytic process. For the second term we have
\begin{eqnarray*}
 ||R_{\varphi_n}^{-1} (y_0/(\tilde \varphi_n)_x^b) - R_{\tilde \varphi_n}^{-1} (y_0/(\tilde \varphi_n)_x^b||_{s-2} &\leq& C_3 ||y_0/(\tilde \varphi_n)_x^b||_{s-1} ||\varphi_n^{-1}-\tilde \varphi_n^{-1}||_{s-1} \\ 
&\leq& C_3^2 ||y_0/(\tilde \varphi_n)_x^b||_{s-1} ||\varphi_n-\tilde \varphi_n||_s
\end{eqnarray*}
which goes to $0$ as $y_0/(\tilde \varphi_n)_x^b$ is bounded in $H^{s-1}$. For the $v_n$-term we have
\[
 ||R_{\tilde \varphi_n}^{-1}((1-\partial_x^2)v_n/(\tilde \varphi_n)_x^b)||_{s-2} \leq C_2 ||(1-\partial_x^2)v_n/(\tilde \varphi_n)_x^b||_{s-2}
\]
which by the definition of $v_n$ goes to zero. Hence we conclude
\begin{multline*}
  \limsup_{n \to \infty} ||R_{\varphi_n}^{-1}\left(y_n/(\varphi_n)_x^b\right)-R_{\tilde \varphi_n}^{-1} \left(\tilde y_n/(\tilde \varphi_n)_x^b \right)||_{s-2} \\=\limsup_{n \to \infty} ||R_{\varphi_n}^{-1}\left((1-\partial_x^2)w_n/(\varphi_n)_x^b\right)-R_{\tilde \varphi_n}^{-1} \left((1-\partial_x^2)w_n/(\tilde \varphi_n)_x^b \right)||_{s-2}
\end{multline*}
The claim is now that the latter two terms have disjoint support. To establish this we estimate $|\varphi_n(x_0)-\tilde \varphi_n(x_0)|$. By the Taylor expansion we have
\[
 \varphi_n = \exp(u_0)+d_{u_0}\exp(w_n) + \int_0^1 (1-t) d_{u_0+tw_n}^2(w_n,w_n) dt
\]
and similarly
\[
 \tilde \varphi_n = \exp(u_0)+d_{u_0}\exp(w_n+v_n) + \int_0^1 (1-t) d_{u_0 +t(w_n+v_n)}^2(w_n+v_n,w_n+v_n) dt
\]
For the latter quadratic term we have 
\begin{multline*}
 d_{u_0 +t(w_n+v_n)}^2(w_n+v_n,w_n+v_n)=d_{u_0+t(w_n+v_n)}^2(w_n,w_n)\\
+2 d_{u_0+t(w_n+v_n)}^2(w_n,v_n) + d_{u_0+t(w_n+v_n)}^2(v_n,v_n)
\end{multline*}
Thus we can write
\[
 \varphi_n - \tilde \varphi_n = d_{u_0}\exp(v_n)+\mathcal R_1 + \mathcal R_2 + \mathcal R_3
\]
where
\[
 \mathcal R_1 = \int_0^1 (1-t) \left(d_{u_0+t(w_n+v_n)}^2(w_n,w_n)-d_{u_0+t(w_n+v_n)}^2(w_n,w_n)\right) dt
\]
and 
\[
 \mathcal R_2 = 2 \int_0^1 (1-t)  d_{u_0+t(w_n+v_n)}^2(w_n,v_n) dt
\]
and
\[
 \mathcal R_3 = \int_0^1 (1-t)  d_{u_0+t(w_n+v_n)}^2(v_n,v_n) dt
\]
For these we have
\[
 ||\mathcal R_1||_\infty \leq C ||\mathcal R_1||_s \leq C K ||v_n||_s ||w_n||_s^2 \leq \frac{1}{n} C K ||v||_s (R/4)^2 \leq \frac{1}{4n} C K R^2 ||v||_s 
\]
and 
\[
 ||\mathcal R_2||_\infty \leq C ||\mathcal R_2||_s \leq 2 C K ||v_n||_s ||w_n||_s \leq \frac{1}{n} C K ||v||_s (R/4) \leq \frac{2}{4n} C K R ||v||_s
\]
and
\[
 ||\mathcal R_3||_\infty \leq C ||\mathcal R_3||_s \leq  C K ||v_n||_s^2 \leq \frac{1}{n} C K ||v||_s (R/4) \leq \frac{1}{4n} C K R ||v||_s
\]
Therefore
\begin{eqnarray*}
 |\varphi(x_0)-\tilde \varphi(x_0)| &\geq &|d_{u_0}\exp(v_n)| - ||\mathcal R_1||_\infty - ||\mathcal R_2||_\infty - ||\mathcal R_3||_\infty \\
&\geq& \frac{1}{n} m ||v||_s - \frac{1}{n} \frac{m}{2} ||v||_s = \frac{m}{2n} ||v||_s
\end{eqnarray*}
The support of $R_{\varphi_n}^{-1}\left((1-\partial_x^2)w_n/(\varphi_n)_x^b\right)$ is contained in $(\varphi_n(x_0)-r_n,\varphi_n(x_0)+r_n)$ taking into account the lipschitz property of $\varphi_n$ with lipschitz constant $L$ and the definition of $w_n$. Analogously the support of $R_{\tilde \varphi_n}^{-1}\left((1-\partial_x^2)w_n/(\tilde \varphi_n)_x^b\right)$ is contained in $(\tilde \varphi_n(x_0)-r_n,\tilde \varphi_n(x_0)+r_n)$. Note that the conditions of Lemma \ref{lemma_vanishing_support1} resp. \ref{lemma_vanishing_support2} are fullfilled (with $s-2$ playing the role of $s$ in the Lemma) as
\[
 r_n \leq |\varphi_n(x_0)-\tilde \varphi_n(x_0)|/4
\]
Thus we have
\begin{multline*}
 \limsup_{n \to \infty} ||R_{\varphi_n}^{-1}\left((1-\partial_x^2)w_n/(\varphi_n)_x^b\right)-R_{\tilde \varphi_n}^{-1} \left((1-\partial_x^2)w_n/(\tilde \varphi_n)_x^b \right)||_{s-2}^2 \\
\geq \limsup_{n \to \infty} \tilde C (||R_{\varphi_n}^{-1}\left((1-\partial_x^2)w_n/(\varphi_n)_x^b\right)||_{s-2}^2+||R_{\tilde \varphi_n}^{-1} \left((1-\partial_x^2)w_n/(\tilde \varphi_n)_x^b \right)||_{s-2}^2)\\
\geq \limsup_{n \to \infty} \tilde C \frac{2}{C^2} ||(1-\partial_x^2)w_n||_{s-2}^2 \geq \limsup_{n \to \infty} \tilde K ||w_n||_s^2 = \tilde K R^2/4
\end{multline*}
So for any $R \leq R_\ast$ we have constructed $(x_n)_{n \geq 1},(\tilde x_n)_{n \geq 1} \subseteq B_R(u_0)$ with $\lim_{n \to \infty} ||x_n-\tilde x_n||_s=0$ and $\limsup_{n \to \infty} ||\Phi(x_n)-\Phi(\tilde x_n)||_s \geq C \cdot R$ for some constant $C > 0$ independent of $R$ showing the claim.
\end{proof}

\appendix

\section{Sobolev spaces with negative indices}\label{appendix_b}

In this section we derive the formulas for the expressions which involve Sobolev spaces with negative indices.

\begin{Lemma}\label{lemma_extend_multiplication}
Let $1/2 < s_1 < 1$ and $-1/2 < s_2 <0$. Then multiplication
\[
H^{s_1}(\R) \times H^{s_1}(\R) \to H^{s_1}(\R),\quad (f,g) \mapsto f \cdot g
\]
extends to a continuous map
\[
 H^{s_1}(\R) \times H^{s_2}(\R) \to H^{s_2}(\R)
\]
where $H^{s_2}(\R)$ denotes the dual of $H^{-s_2}(\R)$.
\end{Lemma}

\begin{proof}
For $f \in H^{s_1}(\R)$, $g \in H^{s_2}(\R)$ we define $f \cdot g \in H^{s_2}(\R)$ by its action on a testfunction $\psi$ as
\[
 \langle f \cdot g, \psi \rangle:=\langle g, f \cdot \psi \rangle
\]
where $\langle \cdot,\cdot \rangle$ denotes the duality pairing of $H^{s_2}(\R)$ and $H^{-s_2}(\R)$. As $f \cdot \psi \in H^{-s_2}(\R)$ this definition makes sense. Further we have
\[
 |\langle g, f \cdot \psi \rangle | \leq ||g||_{s_2} \; ||f \cdot \psi||_{-s_2} \leq C ||g||_{s_2} \; ||f||_{s_1} \; ||\psi||_{-s_2}
\]
where we have used that multiplication
\[
 H^{s_1}(\R) \times H^{-s_2}(\R) \to H^{-s_2}(\R)
\]
is continuous -- see e.g. \cite{composition}. This shows that $f \cdot g \in H^{s_2}(\R)$. But it also shows that
\[
 ||f \cdot g||_{s_2} \leq C ||f||_{s_1} \; ||g||_{s_2}.
\]
Hence the claim.
\end{proof}

For the product we have the following Leibniz rule.

\begin{Lemma}\label{lemma_extend_leibniz}
Let $1/2 < \tilde s < 1$ and $f,g \in H^{\tilde s}(\R)$. Then $f \cdot g \in H^{\tilde s}(\R)$ with
\[
 \partial_x (f \cdot g) = f_x \cdot g + f \cdot g_x
\]
where the subscript refers to differentiation.
\end{Lemma}

\begin{proof}
We just have to prove the formula for the derivative. For test functions $\psi,\phi$ we have
\begin{eqnarray*}
 \langle \partial_x (f \cdot \phi),\psi \rangle &=& -\langle f \cdot \phi,\psi_x \rangle = - \langle f, \phi \cdot \psi_x \rangle = - \langle f, (\phi \cdot \psi)_x - \phi_x \cdot \psi \rangle\\
 &=& \langle \phi \cdot f_x,\psi \rangle + \langle f \cdot \phi_x,\psi \rangle.
\end{eqnarray*}
Therefore we have the following identity in $H^{\tilde s-1}(\R)$
\[
 \partial_x(f \cdot \phi) = f_x \cdot \phi + f \cdot \phi_x.
\] 
Now letting $\phi$ tend to $g$ in $H^{1-\tilde s}(\R)$ we get by Lemma \ref{lemma_extend_multiplication}
\[
 \partial_x(f \cdot g) = f_x \cdot g + f \cdot g_x
\]
as elements in $H^{\tilde s-1}(\R)$.
\end{proof}

Now we extend right translation $R_\varphi$ to negative Sobolev spaces

\begin{Lemma}\label{lemma_extend_right_translation}
Let $s > 3/2$ and $-1/2 < \tilde s <0$. For $\varphi \in \D^s(\R)$ the map
\[
 R_\varphi:H^s(\R) \to H^s(\R),\quad f \mapsto f \circ \varphi
\]
extends to a continuous map $H^{\tilde s}(\R) \to H^{\tilde s}(\R)$.
\end{Lemma}

\begin{proof}
Let $f \in H^{\tilde s}(\R)$ ans $\psi$ a testfunction. We define
\begin{equation}\label{def_right_translation}
 \langle R_\varphi f,\psi \rangle:= \langle f, \frac{\psi \circ \varphi^{-1}}{\varphi_x \circ \varphi^{-1}} \rangle.
\end{equation}
We know -- see e.g. \cite{composition} -- that
\[
 \left|\left| \frac{\psi \circ \varphi^{-1}}{\varphi_x \circ \varphi^{-1}}\right|\right|_{-\tilde s} \leq C ||\psi||_{-\tilde s}
\]
holds. Therefore $R_\varphi f \in H^{\tilde s}(\R)$ and further
\[
 |\langle R_\varphi f,\psi \rangle| = \left| \langle f, \frac{\psi \circ \varphi^{-1}}{\varphi_x \circ \varphi^{-1}}\rangle \right| \leq C ||f||_{\tilde s} \; ||\psi||_{-\tilde s}
\]
which shows that $R_\varphi:H^{\tilde s}(\R) \to H^{\tilde s}(\R)$ is a continuous linear map. 
\end{proof}

\begin{Rem}\label{remark_translation_notation}
We will sometimes write $f \circ \varphi$ instead of $R_\varphi f$ even if $f$ is in a negative Sobolev space.
\end{Rem}

\begin{Rem}\label{remark_weak_continuity}
The composition map 
\[
H^{\tilde s}(\R) \times \D^s(\R) \to H^{\tilde s}(\R),\quad (f,\varphi) \mapsto f \circ \varphi
\]
is not continuous. But as can be seen from \eqref{def_right_translation} it is weakly continuous, i.e 
\[
H^{\tilde s}(\R) \times \D^s(\R) \to \R,\quad (f,\varphi) \mapsto \langle f \circ \varphi,\psi \rangle
\] 
is continuous for any testfunction $\psi$. 
\end{Rem}

There is also the following chain rule

\begin{Lemma}\label{lemma_extend_chain_rule}
Let $s > 3/2$ and $1/2 < \tilde s < 1$. For $\varphi \in \D^s(\R)$ and $f \in H^{\tilde s}(\R)$ we have
\[
 \partial_x (f \circ \varphi) = f_x \circ \varphi \cdot \varphi_x.
\]
as an identity in $H^{\tilde s-1}(\R)$.
\end{Lemma}

\begin{proof}
Let $\psi$ be a testfunction. Then we have
\begin{eqnarray*}
\langle \partial_x (f \circ \varphi),\psi \rangle &=& - \langle f \circ \varphi, \psi_x \rangle = -\langle f, \frac{\psi_x \circ \varphi^{-1}}{\varphi_x \circ \varphi^{-1}} \rangle = -\langle f, \partial_x (\psi \circ \varphi^{-1}) \rangle \\
&=& \langle f_x, \psi \circ \varphi^{-1} \rangle = \langle f_x \circ \varphi, \psi \cdot \varphi_x \rangle = \langle f_x \circ \varphi \cdot \varphi_x,\psi \rangle
\end{eqnarray*}
which shows the claim.
\end{proof}

We will also need the following Lipschitz type estimate.

\begin{Lemma}\label{lemma_lipschitz}
Let $s > 3/2$ and $\varphi_\bullet \in \D^s(\R)$. There is a neighborhood $W \subseteq \D^s(\R)$ of $\varphi_\bullet$ and a constant $C > 0$ with
\[
 ||f\circ \varphi_1 - f \circ \varphi_2||_{s-2} \leq C ||f||_{s-1} ||\varphi_1-\varphi_2||_{s-1}
\]
for all $\varphi_1,\varphi_2 \in W$ and for all $f \in H^{s-1}(\R)$.
\end{Lemma}

\begin{proof}
Let $f \in C_c^\infty(\R)$. We have from the fundamental lemma of calculus
\[
 f(\varphi_2(x))=f(\varphi_1(x)) + \int_0^1 f'\left(\varphi_1(x)+t(\varphi_2(x)-\varphi_1(x))\right) (\varphi_2(x)-\varphi_1(x)) dt
\]
Taking $W$ small enough we can ensure that $\varphi_1 + t(\varphi_2-\varphi_1) \in \D^s(\R)$ for $0 \leq t \leq 1$. We thus see that
\[
 t \to \left(f'\circ (\varphi_1+t(\varphi_2-\varphi_1))\right) \cdot (\varphi_2-\varphi_1)
\]
is continuous from $[0,1]$ to $H^s(\R)$. As evaluation at $x \in \R$ is a continuous linear map on $H^s(\R)$ we have the identity in $H^s(\R)$
\[
 f\circ \varphi_2 - f \circ \varphi_1 = \int_0^1 \left(f'\circ (\varphi_1+t(\varphi_2-\varphi_1))\right) \cdot (\varphi_2-\varphi_1) dt
\]
where the integral is understood as a Riemann integral. We thus have
\[
 ||f \circ \varphi_2-f \circ \varphi_1||_{s-2} \leq \int_0^1 ||f'\circ (\varphi_1+t(\varphi_2-\varphi_1))||_{s-2} ||\varphi_2-\varphi_1||_{s-1} dt
\]
For $W$ small enough we can ensure that
\[
 ||f'\circ (\varphi_1+t(\varphi_2-\varphi_1))||_{s-2} \leq C ||f'||_{s-2}
\]
Thus we get
\[
 ||f \circ \varphi_2-f \circ \varphi_1||_{s-2} \leq C ||f||_{s-1} ||\varphi_2-\varphi_1||_{s-1}
\]
For general $f \in H^{s-1}(\R)$ we get the inequality by taking approximations $f_n$ because $f \mapsto f \circ \varphi$ is continuous in $H^{s-1}(\R)$.
\end{proof}

For the inversion map $\varphi \mapsto \varphi^{-1}$ we have

\begin{Lemma}\label{lemma_lipschitz_inv}
Let $s > 3/2$ and $\varphi_\bullet \in \D^s(\R)$. There is a neighborhood $W \subseteq \D^s(\R)$ of $\varphi_\bullet$ and a constant $C > 0$ with
\[
 ||\varphi_2^{-1}-\varphi_1^{-1}||_{s-1} \leq C ||\varphi_2-\varphi_1||_s
\]
for all $\varphi_1,\varphi_2 \in W$.
\end{Lemma}

\begin{proof}
If $s > 5/2$ then the lemma follows from \cite{composition} where it was shown that $\D^{s+1}(\R) \mapsto \D^s(\R)$ is $C^1$. So it remains to check the case $s \leq 5/2$. Consider first the case $3/2 < s < 2$. We have
\[
 \int_\R |\varphi^{-1} - \tilde \varphi^{-1}|^2 dx = \int_\R |\varphi^{-1}-\varphi^{-1} \circ \varphi \circ \tilde \varphi^{-1}|^2 dx
\]
As by the Sobolev imbedding the $C^1$-norm is bounded by the $H^s$-norm we have
\[
 |\varphi^{-1}(x) - \varphi^{-1}(y)| \leq C |x-y|
\]
with a uniform $C$ in a neighborhood of $\varphi_\bullet$. Hence
\[
 \int_\R |\varphi^{-1} - \tilde \varphi^{-1}|^2 dx \leq C \int_R |x-\varphi(\tilde \varphi^{-1}(x)|^2 dx
\]
By a change of variables we can bound the latter by
\[
 C \int_\R |\tilde \varphi(x)-\varphi(x)|^2 \tilde \varphi_x(x) dx
\]
As $\varphi_x$ is uniformly bounded in a neighborhood of $\varphi_\bullet$ we get
\[
 ||\varphi^{-1} - \tilde \varphi^{-1}||_{L^2} \leq C ||\varphi-\tilde \varphi||_{s-1}
\]
Let us estimate the fractional part. Using the Sobolev-Slobodecki $[\varphi^{-1}-\tilde \varphi^{-1}]_\lambda$ norm with $\lambda=s-1$ we have 
\begin{multline*}
 \int_{\R \times \R} \frac{|[\varphi^{-1}(x)-\tilde \varphi^{-1}(x)]-[\varphi^{-1}(y)-\tilde \varphi^{-1}(y)]|^2}{|x-y|^{1+2\lambda}} dxdy \\= \int_{\R \times \R} \frac{|\Phi(\tilde \varphi^{-1}(x))-\Phi(\tilde \varphi^{-1}(y))|^2}{|x-y|^{1+2\lambda}}dxdy
\end{multline*}
where
\[
\Phi(x)=\varphi^{-1}(\tilde \varphi(x)) - \varphi^{-1}(\varphi(x))
\]
Change of variables gives
\begin{multline*}
 \int_{\R \times \R} \frac{|\Phi(x)-\Phi(y)|^2}{|\tilde \varphi(x)-\tilde \varphi(y)|^{1+2\lambda}} \tilde \varphi_x(x) \tilde \varphi_x(y) dxdy \\ \leq C \int_{\R \times \R} \frac{|\Phi(x)-\Phi(y)|^2}{|x-y|^{1+2\lambda}} \frac{|x-y|^{1+2\lambda}}{|\tilde \varphi(x)-\tilde \varphi(y)|^{1+2\lambda}} dxdy
\end{multline*}
which by the fact that $K |\tilde \varphi(x)-\tilde \varphi(y)| \geq |x-y|$ holds is bounded by
\[
 C \int_{\R \times \R} \frac{|\Phi(x)-\Phi(y)|^2}{|x-y|^{1+2\lambda}} dxdy
\]
By the fundamental lemma of calculus
\[
\Phi(x)=\left(\int_0^1 \frac{dt}{\varphi_x \circ \varphi^{-1}(\varphi(x)+t(\tilde \varphi(x)-\varphi(x)))} \right) \cdot (\tilde \varphi(x)-\varphi(x)) = \Psi(x) \cdot (\tilde \varphi(x)-\varphi(x))
\]
Writing 
\[
 \Phi(x)-\Phi(y)=\Psi(x) \left([\tilde \varphi(x)-\varphi(x)]-[\tilde \varphi(y)-\varphi(y)]\right) + (\Psi(x)-\Psi(y)) (\tilde \varphi(y)-\varphi(y))
\]
Thus we can estimate
\[
 [\Phi]_\lambda \leq \sup_{x \in \R} |\Psi(x)| [\tilde \varphi - \varphi]_\lambda + \sup_{x \in \R} |\tilde \varphi(x)-\varphi(x)| [\Psi]_\lambda
\]
Note that $[\Psi]_\lambda < \infty$ as
\[
 t \mapsto \frac{1}{\varphi_x \circ \varphi^{-1}(\varphi(x)+t(\tilde \varphi(x)-\varphi(x)))}-1
\]
is a continuous path in $H^{s-1}(\R)$. Therefore we get
\[
 [\varphi^{-1}-\tilde \varphi^{-1}]_\lambda \leq C ||\varphi-\tilde \varphi||_{s-1}
\]
Now consider the case $2 \leq s \leq 5/2$. Taking the derivative we get
\[
 (\varphi^{-1}-\tilde \varphi^{-1})'=\varphi_x \circ \varphi^{-1} - \tilde \varphi_x \circ \tilde \varphi^{-1}
\] 
which rewritten is
\[
 \varphi_x \circ \varphi^{-1}-\varphi_x \circ \tilde \varphi^{-1} + \varphi_x \circ \tilde \varphi^{-1} - \tilde \varphi_x \circ \tilde \varphi^{-1}
\]
For the last two terms we have
\[
 ||\varphi_x \circ \tilde \varphi^{-1}-\tilde \varphi_x \circ \tilde \varphi^{-1}||_{s-2} \leq C ||\varphi_x-\tilde \varphi_x||_{s-2} \leq C ||\varphi-\tilde \varphi||_{s-1}
\]
For the first two terms we can argue as in the proof of Lemma \ref{lemma_lipschitz} and write
\[
 ||\varphi_x \circ \varphi^{-1} - \varphi_x \circ \tilde \varphi^{-1}||_{s-2} \leq C ||\varphi_x-1||_{s-1} ||\varphi^{-1}-\tilde \varphi^{-1}||_{s-2}
\]
Hence using the estimate from above for $||\varphi^{-1}-\tilde \varphi^{-1}||_{s-2}$ as $0 \leq s-2 < 1$ we get the claim.
\end{proof}

\section{Inequalities for fractional Sobolev functions}\label{appendix_ineq}
In this section we will establish inequalities of the form 
\[
 ||f+g|| \geq C (||f||_s+||g||_s)
\]
for functions $f,g$ with disjoint support. For fractional $s$ this causes some difficulties as the norm $||\cdot||_s$ is defined in a non-local way. For fixed supports we have

\begin{Lemma}\label{lemma_fixed_supp}
Let $s \in \R$. There is a constant $C > 0$ such that for all $f,g \in C^\infty_c(\R)$ with $\operatorname{supp} f \subseteq (-3,-1)$ and $\operatorname{supp} g \subseteq (1,3)$ we have
\[
 ||f+g||_s^2 \geq C (||f||_s^2 + ||g||_s^2)
\]
\end{Lemma} 

\begin{proof}
We take $\varphi,\psi \in C^\infty_c(\R)$ with $\operatorname{supp}\varphi \subseteq (-3.5,-0.5)$ and $\operatorname{supp}\psi \subseteq (0.5,3.5)$ such that $\left.\varphi\right|_{(-3,-1)} \equiv 1$ and $\left.\psi\right|_{(1,3)} \equiv 1$. We then have
\[
 ||f||_s = ||\varphi (f+g)||_s \leq C_1 ||f+g||_s
\]
and similarly
\[
 ||g||_s = ||\psi (f+g)||_s \leq C_2 ||f+g||_s
\]
giving the desired result.
\end{proof}

In the following we will use the fact that the $H^s$-norm is equivalent to the homogeneous $\dot H^s$-norm if we restrict ourselves to functions with support in a fixed compact $K \subseteq \R$ (see e.g. \cite{bahouri} p. 39). Recall
\[
 ||f||_{\dot H^s}^2 = \int_\R |\xi|^{2s} |\hat f(\xi)|^2 d\xi
\]
We often also use $f^\lambda(x):=f(x/\lambda)$ for which we have the following scaling property
\[
 ||f^\lambda||_{\dot H^s}^2 = \lambda^{1-2s} ||f||_{\dot H^s}^2 
\]
We have

\begin{Lemma}\label{lemma_vanishing_support1}
Let $s \geq 0$. Then there is a constant $C > 0$ with the following property:
For $x,y$ in $\R$ with $0 < r:=|x-y|/4 < 1$  we have
\[
 ||f+g||_s^2 \geq C  (||f||_s^2 + ||g||_s^2)
\]
for all functions $f,g \in C^\infty_c(\R)$ with $\operatorname{supp}f \subseteq (x-r,x+r)$, $\operatorname{supp}g \subseteq (y-r,y+r)$
\end{Lemma}

\begin{proof}
We use the homogeneous norm. Now scaling with $\lambda=(4r)^{-1}$ gives a situation as in Lemma \ref{lemma_fixed_supp}. We have
\[
 ||f+g||_{\dot H^s}^2 = \lambda_n^{2s-1} ||f^{\lambda}+g^{\lambda}||_{\dot H^s}^2 
\]
Now by Lemma \ref{lemma_fixed_supp} we then get
\[
 ||f+g||_{\dot H^s}^2 \geq C \lambda^{2s-1} \left(||f^{\lambda}||_{\dot H^s}^2+||g^{\lambda}||_{\dot H^s}^2 \right)
\]
Scaling back gives
\[
 ||f+g||_{\dot H^s}^2 \geq C (||f||_{\dot H^s}^2+||g||_{\dot H^s}^2)
\]
This establishes the lemma.
\end{proof}

We will encounter Lemma \ref{lemma_vanishing_support1} also for some negative values of $s$. In these cases we will use

\begin{Lemma}\label{lemma_vanishing_support2}
Let $s <0$ and the same situation as in Lemma \ref{lemma_vanishing_support1}. Then we have 
\[
 ||f+g||_s^2 \geq C (||f||_s^2+||g||_s^2)
\]
for all functions $f,g \in C^\infty_c(\R)$ with $\operatorname{supp}f \subseteq (x-r,x+r)$, $\operatorname{supp}g \subseteq (y-r,y+r)$
\end{Lemma}

\begin{proof}
We claim that for functions with support in some fixed compact set $K \subseteq \R$ the homogeneous norm
\[
 ||f||_{\dot H^s}^2 = \int_\R |\xi|^{2s} |\hat f(\xi)|^2 d\xi
\]
is equivalent to the non-homegeneous norm $||\cdot||_s$. One then can argues as in Lemma \ref{lemma_vanishing_support1} by scaling to a fixed situation as in Lemma \ref{lemma_fixed_supp}. So it remains to show the equivalence of the norms. We clearly have $||\cdot||_s \leq ||\cdot||_{\dot H^s}$ since
\[
 \int_\R (1+\xi^2)^s |\hat f(\xi)|^2 d\xi \leq \int_\R \xi^{2s} |\hat f(\xi)|^2 d\xi
\]
For the other direction we use the dual definition of the Sobolev norm
\[
 ||f||_s = \sup_{||g||_{-s} \leq  1} |\langle f,g \rangle|
\]
and analogously for the homogeneous norm. Taking $\psi \in C_c^\infty(\R)$ with $\psi=1$ on $K$ we have for $f$ with support in $K$
\[
 ||f||_{\dot H^s} = \sup_{||g||_{\dot H^{-s}}\leq 1} |\langle f, \psi \cdot g \rangle|
\]
Now note that we have equivalence of the norms $||\cdot||_{-s}$ and $||\cdot||_{\dot H^{-s}}$ for functions with support in some fixed compact. Therefore
\begin{eqnarray*}
 ||f||_{\dot H^s} &=& \sup_{||g||_{\dot H^s} \leq 1} |\langle f, \psi \cdot g \rangle| \leq \sup_{g,||\psi \cdot g||_{-s} \leq C_1} |\langle f, \psi \cdot g \rangle| \\ &\leq& C_1 \sup_{||g||_{-s} \leq 1} |\langle f, g \rangle|
=||f||_{-s}
\end{eqnarray*}
showing the equivalence.
\end{proof}

\bibliographystyle{plain}

\flushleft
\author{
Hasan Inci\\
EPFL SB MATHAA PDE\\
MA C1 627 (B\^atiment MA)\\
Station 8\\
CH-1015 Lausanne\\
Switzerland\\
        {\it email: } {hasan.inci@epfl.ch}
}

\end{document}